\documentclass[11pt]{article}

\usepackage[a4paper,margin=1in]{geometry}
\usepackage{amsmath,amsthm,amssymb,amsfonts}
\usepackage{mathtools}
\usepackage{hyperref}
\usepackage{tikz}
\usetikzlibrary{trees}

\newtheorem{theorem}{Theorem}[section]
\newtheorem{corollary}[theorem]{Corollary}
\newtheorem{definition}[theorem]{Definition}
\newtheorem{remark}[theorem]{Remark}

\newtheorem{proposition}[theorem]{Proposition}
\newtheorem{example}[theorem]{Example}

\newcommand{\E}{\mathcal E}

\title{Translation Monoids and Recursive Evaluation \\ in Finite Binary Algebras}
\author{Volkan Yildiz\\[0.3em]
\texttt{vo1kan@hotmail.co.uk}}
\date{\today}

\begin{document}

\maketitle

\begin{abstract}
Let \(A=(A,\star)\) be a finite binary algebra, not necessarily associative.
For each \(n\geq 1\), every full binary bracketing on \(x_1,\dots,x_n\)
determines an \(n\)-ary term operation on \(A\), and hence an evaluation word
obtained by listing its values on \(A^n\) in lexicographic order. This produces
an \(m^n\times C_{n-1}\) array, where \(m=|A|\) and \(C_{n-1}\) is the
\((n-1)\)st Catalan number. We show that the recursive structure of these
arrays is governed by the translation monoid
\[
T(A)=\langle L_a,R_a:a\in A\rangle\leq A^A,
\qquad
L_a(x)=a\star x,\quad R_a(x)=x\star a.
\]
More precisely, context maps arising from subterms are exactly the elements of
\(T(A)\), so every element of the translation monoid occurs as a recursive
block map. We also prove that rank defines a natural chain of two-sided ideals
in \(T(A)\), that the minimum-rank elements form a minimal nonempty two-sided
ideal, and that Green's \(\mathcal J\)-classes are contained in rank layers.
Finally, we show by example that equal rank does not determine the
\(\mathcal J\)-class in general.
\end{abstract}

\vfill
\noindent\textbf{Keywords:} finite binary algebras, Catalan bracketings, term operations,
evaluation words, translation monoids, context maps, recursive blocks, transformation semigroups,
Green's relations, minimal ideals.

\medskip
\noindent\textbf{AMS 2020 Mathematics Subject Classification:}\\
Primary 20M20, 08A40.\\
Secondary 05A15, 20M17, 08A05.

\newpage

\newcommand
\HRule{\noindent\rule{\linewidth}{0.5pt}}

\HRule
{\small
\begin{flushright}
{\em ``Dayê, bêhna sêvê tê ...”}
\end{flushright}
}
\HRule

\section{Introduction}

Binary bracketings form one of the most classical Catalan families in algebra.
For each \(n\geq 1\), the set of full binary bracketings on
\(x_1,\dots,x_n\) has cardinality \(C_{n-1}\), the \((n-1)\)st Catalan number.
Once a binary operation is fixed, every such bracketing determines an \(n\)-ary
term operation, and hence an associated evaluation word on the underlying
algebra.

The present paper studies these evaluation words for a finite binary algebra
\(A=(A,\star)\). Fixing an ordering of \(A\), each bracketing determines a word
obtained by listing the values of the induced term operation on \(A^n\) in
lexicographic order. In this way one obtains, for each \(n\), a family of words
indexed by the Catalan set of bracketings. Our aim is to understand the
recursive structure of these words.

The present paper is motivated in part by our earlier study of bracketed
Kleene implications, ~\cite{Yildiz2021}, where the Catalan family of bracketings was analysed at
the level of three-valued truth tables and their associated generating
functions. Here we replace that special setting by a general finite binary
algebra and focus on the translation monoid governing recursive context maps.

The key observation is that when a bracketing is decomposed at a subterm, the
corresponding evaluation word decomposes into blocks, and these blocks are
controlled by unary maps obtained by repeatedly combining with fixed constants
along the chain of enclosing binary operations. For example, consider the bracketing
\[
t=(x_1\star x_2)\star x_3
\]
and the subterm \(u=x_2\). After fixing the variables outside \(u\), say
\[
x_1=0,\qquad x_3=1,
\]
the whole term becomes a unary function of the value of \(u\), namely
\[
x\longmapsto (0\star x)\star 1.
\]
This is a composition of a left translation and a right translation:
\[
x\longmapsto (0\star x)\star 1 = R_1\circ L_0(x).
\]
Indeed,
\[
(R_1\circ L_0)(x)=R_1(0\star x)=(0\star x)\star 1.
\]
Hence the block of the evaluation word of \(t\) obtained by fixing
\[
x_1=0,\qquad x_3=1
\]
is obtained from the evaluation word of \(u\) by coordinatewise action of the
map \(R_1\circ L_0\in T(A)\). This leads naturally to the transformation monoid
\[
T(A)=\langle L_a,R_a : a\in A\rangle \leq A^A,
\qquad
L_a(x)=a\star x,\quad R_a(x)=x\star a,
\]
which we call the translation monoid of \(A\).

The present paper lies between two established lines of research. The first is
the study of Catalan bracketings and the term operations they determine; see,
for example, the literature on associative spectra and its later developments
\cite{CsakanyWaldhauser2000,HuangLehtonen2023}. The second is the study of
translation monoids and context-induced unary maps of algebras
\cite{Petkovic2006,Piirainen2013}. Our contribution is to connect these two
 viewpoints through the recursive evaluation arrays of finite binary algebras.
We show that recursive blocks in bracketing-evaluation arrays are governed 
by the translation monoid, thereby turning a Catalan bracketing problem into 
a natural transformation-semigroup problem.

Our main structural result shows that if \(u\) is a subterm of a bracketing
\(t\), then after fixing all variables outside \(u\), the induced unary map from
the value of \(u\) to the value of \(t\) belongs to \(T(A)\). More precisely,
it is a composition of left and right translations determined by the enclosing
chain from \(u\) to \(t\). As a consequence, every recursive block in every
bracketing-evaluation array is obtained by coordinatewise action of an element
of \(T(A)\) on the evaluation word of a subterm.

This identifies the translation monoid as the natural algebraic object
governing recursive evaluation in finite binary algebras. It also opens a
semigroup-theoretic perspective on the problem: one may ask how the internal
structure of \(T(A)\), including its ideals and Green classes, reflects the recursive 
complexity of the associated evaluation arrays. In the present paper
we focus on the foundational step, namely the passage from subterms to
translations.

The paper is organised as follows. In Section~2 we introduce bracketing
evaluation arrays, give a concrete example, and define the translation monoid.
We then formalise subterm context maps and prove that they all belong to
\(T(A)\). Section~3 shows that the realised context maps coincide with the whole
translation monoid and develops the first semigroup-theoretic consequences of
this construction.

\section{Bracketings, evaluation words, and translation monoids}

For each integer \(n\geq 1\), let \(\mathcal B_n\) denote the set of all full
binary bracketings on the variables
\[
x_1,\dots,x_n,
\]
with the order of variables preserved. Thus \(\mathcal B_n\) may be identified
with the set of all binary terms obtained from \(x_1,\dots,x_n\) by inserting
parentheses in all possible ways, and
\[
|\mathcal B_n|=C_{n-1},
\]
where \(C_{n-1}\) is the \((n-1)\)st Catalan number.

\begin{definition}
Let \(A=(A,\star)\) be a finite binary algebra, (not assumed to be associative), and fix once and for all an
ordering
\[
A=\{a_1,\dots,a_m\},
\qquad m=|A|.
\]

For \(t\in\mathcal B_n\), let
\[
t^A:A^n\to A
\]
denote the \(n\)-ary term operation determined by \(t\) in \(A\).

The \emph{evaluation word} of \(t\) over \(A\), denoted \(w_t\), is the word of
length \(m^n\) obtained by listing the values of \(t^A\) on all tuples of
\(A^n\) in lexicographic order relative to the fixed ordering of \(A\). The
family
\[
\mathcal E_n(A)=(w_t)_{t\in \mathcal B_n}
\]
will be called the \emph{\(n\)th bracketing-evaluation array} of \(A\).
\end{definition}

Thus \(\mathcal E_n(A)\) may be viewed either as an indexed family of words over
the alphabet \(A\), or as an \(m^n\times C_{n-1}\) array whose columns are
indexed by binary bracketings. The combinatorics of bracketings supplies the
Catalan indexing, while the algebra \(A\) determines the resulting words and the
relations among them.

To understand the recursive structure of the words \(w_t\), it is useful to
look at what happens when one focuses on a subterm \(u\) inside a bracketing
\(t\). After fixing values for all variables outside \(u\), the whole term \(t\)
becomes a unary function of the value of \(u\). Thus each occurrence of a
subterm gives rise to a map \(A\to A\). The simplest such maps are obtained by
multiplying a variable value on the left or on the right by a fixed element of
\(A\), namely
\[
L_a(x)=a\star x,\qquad R_a(x)=x\star a.
\]

\begin{definition}
For each \(a\in A\), define maps \(L_a,R_a:A\to A\) by
\[
L_a(x)=a\star x,
\qquad
R_a(x)=x\star a
\qquad (x\in A).
\]
The submonoid of the full transformation monoid \(A^A\) generated by these maps
will be denoted by
\[
T(A)=\langle L_a,R_a : a\in A\rangle \leq A^A
\]
and called the \emph{translation monoid} of \(A\).
\end{definition}

We next formalise the unary maps arising from subterm occurrences.

\begin{definition}
Let \(t\in \mathcal B_n\), and let \(u\) be a subterm occurrence of \(t\). Let
the variables occurring in \(u\) be regarded as internal to \(u\), and all other
variables of \(t\) as external to \(u\). After fixing values of the external
variables, the term \(t\) determines a unary map
\[
\kappa_{t,u,\mathbf a}:A\to A,
\]
where \(\mathbf a\) records the chosen values of the external variables, by
evaluating \(u\) as a single input variable and then evaluating the whole term
\(t\). We call \(\kappa_{t,u,\mathbf a}\) the \emph{context map} of the
occurrence \(u\) in \(t\) relative to the external assignment \(\mathbf a\).
\end{definition}

For a subterm occurrence \(u\) of a bracketing \(t\), let \(x_{i_1},\dots,x_{i_k}\) be the
variables occurring in \(u\), in their inherited left-to-right order. By relabelling these
variables as \(x_1,\dots,x_k\), the occurrence \(u\) determines a bracketing
\(\widetilde{u}\in B_k\). We define the evaluation word of the occurrence \(u\) to be
\[
w_u:=w_{\widetilde{u}}.
\]
Equivalently, \(w_u\) is obtained by listing the values of the induced term operation of \(u\)
as its internal variables vary over \(A^k\) in lexicographic order.

\begin{example}
The following example illustrates both the evaluation words and the
subterm-to-root context maps. Let
\[
A=\{0,1\},
\]
equipped with the binary operation \(\star\) given by
\[
\begin{array}{c|cc}
\star & 0 & 1\\ \hline
0 & 1 & 1\\
1 & 1 & 0
\end{array}
\]
and fix the order \(0<1\) on \(A\).
For \(n=3\), the set \(\mathcal B_3\) consists of the two bracketings
\[
t_1=x_1\star(x_2\star x_3),
\qquad
t_2=(x_1\star x_2)\star x_3.
\]
Listing the elements of \(A^3\) in lexicographic order,
\[
(0,0,0),(0,0,1),(0,1,0),(0,1,1),
(1,0,0),(1,0,1),(1,1,0),(1,1,1),
\]
we obtain
\[
w_{t_1}=11110001,
\qquad
w_{t_2}=10101011.
\]
Thus
\[
\mathcal E_3(A)=\bigl(w_{t_1},w_{t_2}\bigr)
=\bigl(11110001,\;10101011\bigr),
\]
so the two bracketings induce different ternary term operations on \(A\).
\newpage

Now let \(u=x_2\), viewed as a subtree occurrence in each bracketing. The two
bracketing trees are shown below, with the path from \(u\) to the root marked in
blue.

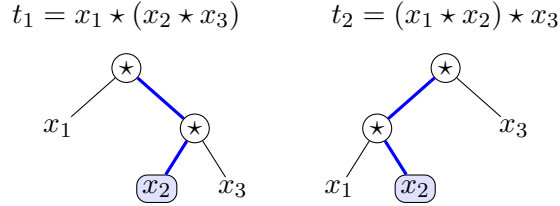
\begin{center}
\begin{tikzpicture}[
  baseline=(current bounding box.center),
  op/.style={circle,draw,inner sep=1.2pt},
  leaf/.style={inner sep=1pt},
  marked/.style={draw,rounded corners,fill=blue!12,inner sep=2pt}
]

\node at (-2.1,2.2) {$t_1=x_1\star(x_2\star x_3)$};

\node[op]   (r1) at (-2.1,1.5) {$\star$};
\node[leaf] (a1) at (-3.0,0.7) {$x_1$};
\node[op]   (b1) at (-1.2,0.7) {$\star$};
\node[marked] (u1) at (-1.7,-0.1) {$x_2$};
\node[leaf] (c1) at (-0.7,-0.1) {$x_3$};

\draw (r1) -- (a1);
\draw (r1) -- (b1);
\draw (b1) -- (u1);
\draw (b1) -- (c1);
\draw[blue,very thick] (u1) -- (b1) -- (r1);

\node at (2.1,2.2) {$t_2=(x_1\star x_2)\star x_3$};

\node[op]   (r2) at (2.1,1.5) {$\star$};
\node[op]   (b2) at (1.2,0.7) {$\star$};
\node[leaf] (c2) at (3.0,0.7) {$x_3$};
\node[leaf] (a2) at (0.7,-0.1) {$x_1$};
\node[marked] (u2) at (1.7,-0.1) {$x_2$};

\draw (r2) -- (b2);
\draw (r2) -- (c2);
\draw (b2) -- (a2);
\draw (b2) -- (u2);
\draw[blue,very thick] (u2) -- (b2) -- (r2);

\end{tikzpicture}
\end{center}

Fix the external variables by
\[
x_1=0,
\qquad
x_3=1.
\]
Then the induced context maps are
\[
\kappa_{t_1,u}(x)=0\star(x\star 1)=L_0\circ R_1(x),
\]
and
\[
\kappa_{t_2,u}(x)=(0\star x)\star 1=R_1\circ L_0(x).
\]

Since
\[
L_0(x)=0\star x=1
\quad\text{for all }x\in A,
\]
and
\[
R_1(0)=0\star 1=1,\qquad R_1(1)=1\star 1=0,
\]
we obtain
\[
\kappa_{t_1,u}(0)=1,\qquad \kappa_{t_1,u}(1)=1,
\]
so \(\kappa_{t_1,u}\) is the constant map \(1\), while
\[
\kappa_{t_2,u}(0)=0,\qquad \kappa_{t_2,u}(1)=0,
\]
so \(\kappa_{t_2,u}\) is the constant map \(0\).

Finally, the evaluation word of the subtree \(u=x_2\) is
\[
w_u=01.
\]
Hence the corresponding recursive blocks are
\[
\kappa_{t_1,u}(w_u)=11,
\qquad
\kappa_{t_2,u}(w_u)=00.
\]

This example shows, in a single setting, how bracketings determine evaluation
words and how subtree-to-root paths determine context maps as compositions of
left and right translations.
\end{example}

The main point is that these context maps do not range arbitrarily over all
unary transformations of \(A\). They are already forced to lie in the
translation monoid.

\begin{theorem}\label{thm:context-in-translation-monoid}
Let \(A=(A,\star)\) be a finite binary algebra, let \(t\in \mathcal B_n\), and
let \(u\) be a subterm occurrence of \(t\). For every choice of values of the
variables external to \(u\), the associated context map
\[
\kappa_{t,u,\mathbf a}:A\to A
\]
belongs to the translation monoid \(T(A)\). More precisely,
\(\kappa_{t,u,\mathbf a}\) is a composition of left and right translations,
with one factor for each enclosing binary operation on the chain from \(u\) to
\(t\).
\end{theorem}

\begin{proof}
We argue by induction on the distance from \(u\) to \(t\) in the chain of
enclosing subterm occurrences.

If this distance is \(0\), then \(u=t\). In that case the context map is simply
the identity map on \(A\), since the value inserted at \(u\) is already the
value of the whole term \(t\). As \(T(A)\) is a monoid, it contains the
identity transformation, so the claim holds.

Now assume the statement holds whenever the distance is at most \(d\), and let
\(u\) be at distance \(d+1\) from \(t\). Let \(v\) be the parent subterm
occurrence of \(u\). Then exactly one of the following two cases occurs:
\[
v=s\star u
\qquad\text{or}\qquad
v=u\star s,
\]
where \(s\) is the sibling subterm occurrence of \(u\) inside \(v\).

Fix an assignment of values to all variables external to \(u\). In particular,
all variables occurring in \(s\) are fixed, so the value of \(s\) under this
assignment is a constant \(c\in A\).

Suppose first that \(v=s\star u\). Then whenever the variable input to the
context map has value \(x\in A\), the value at \(v\) is
\[
c\star x=L_c(x).
\]
From \(v\) upward to the whole term \(t\), the remaining enclosing context has
distance \(d\). By the induction hypothesis, the induced map from the value at
\(v\) to the value of \(t\) is some element \(g\in T(A)\). Hence the total
context map is
\[
\kappa_{t,u,\mathbf a}=g\circ L_c \in T(A).
\]

The case \(v=u\star s\) is analogous. The value at \(v\) is then
\[
x\star c = R_c(x),
\]
and by the induction hypothesis the remaining enclosing context contributes some
map \(g\in T(A)\). Therefore
\[
\kappa_{t,u,\mathbf a}=g\circ R_c \in T(A).
\]

In both cases the context map lies in \(T(A)\), and each step from \(u\) to
\(t\) contributes exactly one left or right translation. This completes the
induction.
\end{proof}

The preceding theorem shows that the recursive structure of bracketing
evaluations is mediated by the translation monoid. To express the immediate
consequence for evaluation words, we let unary maps act coordinatewise on words:
if \(f:A\to A\) and \(v=v_1v_2\cdots v_r\in A^r\), then
\[
f(v):=f(v_1)f(v_2)\cdots f(v_r).
\]

\begin{corollary}\label{cor:recursive-blocks}
Let \(t\in \mathcal B_n\), and let \(u\) be a subterm occurrence of \(t\). Every
block of the evaluation word \(w_t\) obtained by fixing all variables external
to \(u\) has the form
\[
f(w_u)
\]
for some \(f\in T(A)\). In particular, every recursive block arising in any
bracketing-evaluation array \(\mathcal E_n(A)\) is obtained from the evaluation
word of a subterm by the coordinatewise action of an element of the translation
monoid.
\end{corollary}

\begin{proof}
Fix values of all variables external to \(u\). As the internal variables of
\(u\) run through all possible values in lexicographic order, the values of the
subterm \(u\) form the word \(w_u\). By
Theorem~\ref{thm:context-in-translation-monoid}, the corresponding context map
from the value of \(u\) to the value of \(t\) is some transformation
\(f\in T(A)\). Applying this transformation coordinatewise to the word \(w_u\)
produces exactly the corresponding block of \(w_t\). The final statement follows
immediately.
\end{proof}

\begin{remark}
Corollary~\ref{cor:recursive-blocks} identifies \(T(A)\) as the basic algebraic
object underlying recursive evaluation in finite binary algebras. It is
therefore natural to ask how the semigroup-theoretic structure of \(T(A)\),
including its ideals and Green classes, reflects the recursive complexity of the
families \(\E_n(A)\).
\end{remark}

\newpage

\section{Realised context maps and first semigroup consequences}

We now make precise the relation between recursive subterm contexts and the
translation monoid. The main point is that no gap remains: every element of the
translation monoid is realised by a context map.

\begin{definition}
Let \(A=(A,\star)\) be a finite binary algebra. Denote by \(\mathcal C(A)\) the
set of all unary maps \(A\to A\) arising as context maps
\[
\kappa_{t,u,\mathbf a}:A\to A
\]
from subterm occurrences \(u\) in bracketings \(t\), under arbitrary external
assignments \(\mathbf a\).
\end{definition}

By Theorem~\ref{thm:context-in-translation-monoid}, one already has
\[
\mathcal C(A)\subseteq T(A).
\]
The reverse inclusion also holds.
\begin{proposition}\label{prop:C-equals-T}
For every finite binary algebra \(A\),
\[
\mathcal C(A)=T(A).
\]
\end{proposition}

\begin{proof}
The inclusion \(\mathcal C(A)\subseteq T(A)\) is exactly
Theorem~\ref{thm:context-in-translation-monoid}.

Conversely, let \(f\in T(A)\). Then \(f\) is a composition of left and right
translations, say
\[
f=\varphi_k\circ\cdots\circ \varphi_1,
\qquad
\varphi_i\in \{L_a,R_a:a\in A\}.
\]
We construct a one-hole context \(Q[z]\) recursively. Set
\[
Q_0[z]=z.
\]
If \(\varphi_i=L_a\), define
\[
Q_i[z]=a\star Q_{i-1}[z].
\]
If \(\varphi_i=R_a\), define
\[
Q_i[z]=Q_{i-1}[z]\star a.
\]
Then, by induction on \(i\), the unary map induced by \(Q_i[z]\) is
\[
\varphi_i\circ\cdots\circ\varphi_1.
\]
Hence \(Q_k[z]\) induces \(f\).

Finally, replace each constant occurring in \(Q_k[z]\) by a fresh external
variable, and let \(u\) be a single variable in place of \(z\). Under the
external assignment sending those fresh variables back to the chosen constants,
the resulting subterm context map is exactly \(f\). Therefore
\[
f\in \mathcal C(A).
\]
So \(T(A)\subseteq \mathcal C(A)\), and equality follows.
\end{proof}

Thus the recursive bracketing process sees the whole translation monoid, and not
merely a distinguished subset of it.

\newpage

\begin{corollary}\label{cor:every-element-realized}
For every \(f\in T(A)\), there exist an integer \(n\geq 1\), a bracketing
\(t\in \mathcal B_n\), and a subterm occurrence \(u\) of \(t\) such that a
recursive block of \(w_t\) is equal to
\[
f(w_u).
\]
In particular, every element of \(T(A)\) occurs as a recursive block map in
some bracketing-evaluation array.
\end{corollary}

\begin{proof}
Let \(f\in T(A)\). By Proposition~\ref{prop:C-equals-T}, there exist a
bracketing \(t\), a subterm occurrence \(u\) of \(t\), and an external
assignment such that the associated context map is exactly \(f\). By
Corollary~\ref{cor:recursive-blocks}, the corresponding recursive block of
\(w_t\) is therefore
\[
f(w_u).
\]
Hence \(f\) occurs as a recursive block map in the bracketing-evaluation array
containing \(w_t\).
\end{proof}

\begin{corollary}\label{cor:all-ideals-and-Jclasses-realized}
Every nonempty ideal of \(T(A)\) and every \(\mathcal J\)-class of \(T(A)\)
contains an element that occurs as a recursive block map in some
bracketing-evaluation array.
\end{corollary}

\begin{proof}
Let \(X\) be either a nonempty ideal of \(T(A)\) or a \(\mathcal J\)-class of
\(T(A)\). Choose \(f\in X\). By
Corollary~\ref{cor:every-element-realized}, the element \(f\) occurs as a
recursive block map in some bracketing-evaluation array. Therefore \(X\)
contains such an element.
\end{proof}

We now record the first semigroup-theoretic consequences.

\begin{definition}
For \(f\in T(A)\), define the rank of \(f\) by
\[
\rho(f)=|f(A)|.
\]
\end{definition}

\begin{proposition}\label{prop:rank-monotonicity}
For all \(f,g,h\in T(A)\),
\[
\rho(g\circ f\circ h)\leq \rho(f).
\]
\end{proposition}
\begin{proof}
Since \(h(A)\subseteq A\), we have
\[
f(h(A))\subseteq f(A).
\]
Applying \(g\) to both sides yields
\[
g(f(h(A)))\subseteq g(f(A)).
\]
Hence
\[
(g\circ f\circ h)(A)=g(f(h(A)))\subseteq g(f(A)).
\]
Therefore
\[
\rho(g\circ f\circ h)=|(g\circ f\circ h)(A)|
\leq |f(A)|
=\rho(f).
\]
\end{proof}

\begin{corollary}\label{cor:rank-ideals}
For each integer \(r\) with \(1\leq r\leq |A|\), the set
\[
I_r=\{f\in T(A):\rho(f)\leq r\}
\]
is a two-sided ideal of \(T(A)\).
\end{corollary}

\begin{proof}
Let \(f\in I_r\) and \(g,h\in T(A)\). By
Proposition~\ref{prop:rank-monotonicity},
\[
\rho(gfh)\leq \rho(f)\leq r.
\]
So \(gfh\in I_r\), proving that \(I_r\) is a two-sided ideal.
\end{proof}

\begin{proposition}\label{prop:min-rank-single-J-class}
Let \(f,g\in T(A)\) satisfy
\[
\rho(f)=\rho(g)=r_{\min},
\]
where
\[
r_{\min}=\min\{\rho(h):h\in T(A)\}.
\]
Then
\[
f\,\mathcal J\, g.
\]
\end{proposition}

\begin{proof}
We first show that
\[
\rho(gf)=r_{\min}.
\]
Indeed, by Proposition~\ref{prop:rank-monotonicity},
\[
\rho(gf)\leq \rho(f)=r_{\min}.
\]
Since \(gf\in T(A)\) and \(r_{\min}\) is the minimum rank of an element of \(T(A)\), 
it follows that \(\rho(gf)=r_{\min}\). Similarly,
\[
\rho(fg)=r_{\min}.
\]
Now
\[
\operatorname{Im}(gf)\subseteq \operatorname{Im}(g),
\]
and both sets have cardinality \(r_{\min}\). Hence
\[
\operatorname{Im}(gf)=\operatorname{Im}(g).
\]
It follows that the restriction of \(gf\) to \(\operatorname{Im}(g)\) is a
surjective self-map of the finite set \(\operatorname{Im}(g)\), hence a
bijection. Therefore some power \((gf)^n\) acts as the identity on \(\operatorname{Im}(g)\).
Since \(g(A)=\operatorname{Im}(g)\), we obtain
\[
(gf)^n g = g.
\]
Thus
\[
g=(gf)^n g = g\bigl(f(gf)^{\,n-1}g\bigr),
\]
so \(g\in T(A)fT(A)\).\\\\
By symmetry, interchanging the roles of \(f\) and \(g\), we also obtain
\[
f\in T(A)gT(A).
\]
Hence
\[
T(A)fT(A)=T(A)gT(A),
\]
that is,
\[
f\,\mathcal J\, g.
\]
\end{proof}

\begin{corollary}\label{cor:minimal-ideal}
Let
\[
r_{\min}=\min\{\rho(f):f\in T(A)\}.
\]
Then
\[
K(T(A))=\{f\in T(A):\rho(f)=r_{\min}\}
\]
is a minimal nonempty two-sided ideal of \(T(A)\).
\end{corollary}

\begin{proof}
By Corollary~\ref{cor:rank-ideals}, the set \(K(T(A))\) is a two-sided ideal.
By Proposition~\ref{prop:min-rank-single-J-class}, all elements of \(K(T(A))\)
belong to the same \(\mathcal J\)-class. Hence \(K(T(A))\) is a single \(\mathcal J\)-class.
 Since every two-sided ideal is a union of \(\mathcal J\)-classes, it follows that \(K(T(A))\) is a minimal
nonempty two-sided ideal. 
\end{proof}

\begin{proposition}\label{prop:J-same-rank}
If \(f,g\in T(A)\) satisfy
\[
f\,\mathcal J\, g,
\]
then
\[
\rho(f)=\rho(g).
\]
\end{proposition}

\begin{proof}
Since \(f\,\mathcal J\, g\), there exist \(a,b,c,d\in T(A)\) such that
\[
f=agb,
\qquad
g=cfd.
\]
Applying Proposition~\ref{prop:rank-monotonicity} twice gives
\[
\rho(f)\leq \rho(g)
\qquad\text{and}\qquad
\rho(g)\leq \rho(f).
\]
Hence \(\rho(f)=\rho(g)\).
\end{proof}

So each \(\mathcal J\)-class is contained in a single rank layer.

The converse, however, fails in general.

\begin{example}\label{ex:rank-not-J}
Let \(A=\{0,1,2\}\) with binary operation given by
\[
\begin{array}{c|ccc}
\star & 0 & 1 & 2\\ \hline
0 & 1 & 0 & 0\\
1 & 1 & 1 & 1\\
2 & 0 & 1 & 2
\end{array}
\]
Then
\[
L_0=(1,0,0),\qquad L_1=(1,1,1),\qquad L_2=(0,1,2),
\]
and
\[
R_0=(1,1,0),\qquad R_1=(0,1,1),\qquad R_2=(0,1,2),
\]
where maps are written as tuples \((f(0),f(1),f(2))\). A direct calculation shows that:
\[
T(A)=\{
(0,1,2),\,
(1,1,1),\,
(0,0,0),\,
(1,0,0),\,
(1,1,0),\,
(0,1,1),\,
(0,0,1)
\}.
\]
Its \(\mathcal J\)-classes are
\[
\{(0,1,2)\},
\qquad
\{(1,1,1),(0,0,0)\},
\qquad
\{(1,0,0),(0,1,1)\},
\qquad
\{(1,1,0),(0,0,1)\}.
\]
The last two classes both consist of elements of rank \(2\), so equality of
rank does not determine the \(\mathcal J\)-class.
\end{example}
Finally, there is no general ``eventual collapse'' into lower-rank ideals.
\begin{proposition}\label{prop:no-eventual-collapse}
There is no general theorem forcing deeper recursive levels to enter lower-rank
ideals of \(T(A)\).
\end{proposition}

\begin{proof}
Let \(A=G\) be a finite group, written multiplicatively, and let \(\star\) be
the group multiplication. For each \(a\in G\), both \(L_a\) and \(R_a\) are
permutations of \(G\). Hence every element of \(T(G)\) is a permutation, so
every element has rank \(|G|\). Thus \(T(G)\) has no proper nonempty rank ideal,
and no recursive block map can be forced into a lower-rank layer by increasing
depth.
\end{proof}

For the translation-monoid programme developed here, the basic structural
questions are resolved as follows. The realised context maps are exactly the
elements of \(T(A)\), and hence every element, every nonempty ideal, and every
\(\mathcal J\)-class of \(T(A)\) is realised by recursive block maps. The rank
filtration
\[
I_r=\{f\in T(A):\rho(f)\le r\}
\]
consists of two-sided ideals, and the minimum-rank layer forms a minimal
nonempty two-sided ideal of \(T(A)\). Moreover,
\[
f\,\mathcal J\, g \implies \rho(f)=\rho(g),
\]
although the converse fails in general: equal rank need not determine the
\(\mathcal J\)-class. Finally, there is no general eventual descent of recursive
block maps into lower-rank ideals.

\vspace{36em}

\HRule
{\small
\begin{flushleft}
Oh, come with old Khayyam, and leave the Wise \\
To talk; one thing is certain, that Life flies; \\
     One thing is certain, and the Rest is Lies;  \\
The Flower that once has blown for ever dies.\\
\end{flushleft}
\begin{flushleft}
{\em O. Khayyam}
\end{flushleft}
}
\HRule

\end{document}